\documentclass{article}
\usepackage{amsmath}
\usepackage{amsfonts}

\newtheorem{theorem}{Theorem}

\newtheorem{example}[theorem]{Example}

\newtheorem{proposition}[theorem]{Proposition}

\newenvironment{proof}[1][Proof]{\noindent\textbf{#1.} }{\ \rule{0.5em}{0.5em}}
\input{tcilatex}
\begin{document}

\title{\textbf{Non-spurious solutions to discrete boundary value problems
through variational methods}}
\author{ Marek Galewski \thanks{%
Lodz University of Technology, Poland, email: marek.galewski@p.lodz.pl} \and %
Ewa Schmeidel\thanks{%
University of Bialystok, Poland, email: eschmeidel@math.uwb.edu.pl}}
\maketitle
\date{}

\begin{abstract}
\noindent Using direct variational method we consider the existence of
non-spurious solutions to the following Dirichlet problem $\ddot{x}\left(
t\right) =f\left( t,x\left( t\right) \right) $, $x\left( 0\right) =x\left(
1\right) =0 $ where $f:\left[ 0,1\right] \times \mathbb{R} \rightarrow 
\mathbb{R}$ is a jointly continuous function convex in $x$ which does not
need to satisfy any further growth conditions.


\end{abstract}



\section{Introduction}

In this note we consider non-spurious solutions by using a critical point
theory to the following Dirichlet problem 
\begin{equation}
\begin{array}{l}
\ddot{x}\left( t\right) =f\left( t,x\left( t\right) \right) \\ 
x\left( 0\right) =x\left( 1\right) =0%
\end{array}
\label{par}
\end{equation}
where $f:\left[ 0,1\right] \times \mathbb{R} \rightarrow \mathbb{R}$ is a
jointly continuous function. Further we will make precise what is meant by
the solutions to (\ref{par}).

The existence of non-spurious solutions is very important for the
applications since in such a case one can approximate solutions to (\ref{par}%
) with a sequence of solutions to a suitably chosen family of discrete
problems and one is sure that this approximation converges to the solution
of the original problem, see \cite{kelly}. There are many ways in which a
boundary value problem can be discretized and the existence and multiplicity
theory on difference equations is very vast, see for example \cite{agrawal}, 
\cite{candito1}, \cite{guojde}, \cite{MRT}. However, as underlined by
Agarwal, \cite{agarvalpaper}, there are no clear relations between
continuous problems and their discretization which means that both problems
can be solvable, but the approximation approaches nothing but the solution
to the continuous problem or else, the discrete problem is solvable and the
continuous one is not or the other way round. Let us recall his examples:

\begin{example}
The continuous problem $\ddot{x}(t) + \frac{\pi^2}{n^2}x(t)=0$, $x(0)=x(n)=0$
has an infinite number of solutions $x(t)= c \sin \frac{\pi t}{n}$ ($c$ is
arbitrary) whereas its discrete analogue $\Delta^2x(k)+\frac{\pi^2}{n^2}%
x(k)=0$, $x(0)=x(n)=0$ has only one solution $x(k)\equiv 0$. The problem $%
\ddot{x}(t)+\frac{\pi^2}{4n^2}x(t)=0$, $x(0)=0$, $x(n)=1$ has only one
solution $x(t)=\sin\frac{\pi t}{2n}$, and its discrete analogue $\Delta^2
x(k)+ \frac{\pi^2}{4n^2}x(k)=0$, $x(0)=0$, $x(n)=1$ also has one solution.
The continuous problem $\ddot{x}(t)+4\sin^2\frac{\pi}{2n}x(t)=0$, $x(0)=0$, $%
x(n)=\varepsilon \neq 0$ has only one solution $x(t)= \varepsilon \frac{\sin[%
(2 \sin\frac{\pi}{2n})t]}{\sin[(2 \sin\frac{\pi}{2n})n]}$, whereas its
discrete analogue $\Delta^2x(k)+4\sin^2\frac{\pi}{2n}x(k)=0$, $x(0)=0$, $%
x(n)=\varepsilon \neq 0$ has no solution.
\end{example}

Thus, the nature of the solution changes when a continuous boundary value
problem is being discretized. Moreover, two-point boundary value problems
involving derivatives lead to multipoint problems in the discrete case.

The above remarks and examples show that steal it is important to consider
both continuous and discrete problems simultaneously and investigate
relation between solutions which is the key factor especially when the
existence part follows by standard techniques.

There have been some research in this case addressing mainly problems whose
solutions where obtained by the fixed point theorems and the method of lower
and upper solutions, \cite{rech1}, \cite{rachunkowa2}, \cite{thomsontisdell}%
. In this submission we are aiming at using critical point theory method,
namely the direct method of the calculus of variations (see for example \cite%
{Ma} for a nice introduction to this topic) in order to show that in this
setting one can also obtain suitable convergence results. The advance over
works mentioned is that we can have better growth conditions imposed on $f$
at the expense of not putting derivative of $x$ in $f$. As expected we will
have to get the uniqueness of solutions for the associated discrete problem,
which is not always easy to be obtained, see \cite{uni2}.

In \cite{kelly} following \cite{gaines}, it is suggested which family of
difference equations for $n\in \mathbb{N}$ is to be chosen when
approximating problem (\ref{par}). For $a$, $b$ such that $a<b<\infty $, $%
a\in \mathbb{N}\cup \{0\}$, $b\in \mathbb{N}$ we denote $\mathbb{N}%
(a,b)=\{a,a+1,...,b-1,b\}$. For a fixed $n\in \mathbb{N}$ the nonlinear
difference equation with Dirichlet boundary conditions is given as follows
for $k\in \mathbb{N}(0,n-1)$ 
\begin{equation}
\Delta ^{2}x\left( k-1\right) =\frac{1}{n^{2}}f\left( \frac{k}{n},x\left(
k\right) \right) ,\,\,\,x\left( 0\right) =x\left( n\right) =0.
\label{diffequ}
\end{equation}%
Here $\Delta $ is the forward difference operator, i.e. $\Delta x\left(
k-1\right) =x\left( k\right) -x\left( k-1\right) $ and we see that $\Delta
^{2}x\left( k-1\right) =x\left( k+1\right) -2x\left( k\right) +x\left(
k-1\right) $. Assume that both continuous boundary value problem (\ref{par})
and for each fixed $n\in \mathbb{N}$ discrete boundary value problem (\ref%
{diffequ}) are uniquely solvable by, respectively $x$ and $x^{n}=\left(
x^{n}(k)\right) $. Moreover, let there exist two constants $Q,N>0$
independent of $n$ and such that 
\begin{equation}
n|\Delta x^{n}(k-1)|\leq Q\text{ and }|x^{n}(k)|\leq N  \label{ewa1}
\end{equation}%
for all $k\in \mathbb{N}(0,n)$ and all $n\geq n_{0}$, where $n_{0}$ is fixed
(and arbitrarily large). Lemma 9.2. from \cite{kelly} says that for some
subsequence $x^{n_{m}}=\left( x^{n_{m}}(k)\right) $ of $x^{n}$ it holds 
\begin{equation}
\lim_{m\rightarrow \infty }\max_{0\leq k\leq n_{m}}\left\vert
x^{n_{m}}\left( k\right) -x\left( \frac{k}{n_{m}}\right) \right\vert =0.
\label{ewa2}
\end{equation}%
In other words, this means that the suitable chosen discretization
approaches the given continuos boundary value problem. Such solutions to
discrete BVPs are called non-spurious in contrast to spurious ones which
either diverge or else converge to anything else but the solution to a given
continuous Dirichlet problem.

\section{Non spurious solutions for (\protect\ref{par})}

\subsection{The continuous problem}

In the existence part we apply variational methods. This means that with
problem under consideration we must associate the Euler action functional,
prove that this functional is weakly lower semicontinuous in a suitable
function space, coercive and at least G\^{a}teaux differentiable. Given this
three conditions one knows that at least a weak solution to problem under
consideration exists whose regularity can further be improved with known
tools. Such scheme, commonly used within the critical point theory is well
described in the first chapters of \cite{Ma}.

The solutions to (\ref{par}) will be investigated in the space $%
H_{0}^{1}\left( 0,1\right) $ consisting of absolutely continuous functions
satisfying the boundary conditions and with a.e. derivative being integrable
with square. Such a solution is called a weak one, i.e. a function $x\in
H_{0}^{1}\left( 0,1\right) $ is a weak $H_{0}^{1}\left( 0,1\right) $
solution to (\ref{par}), if 
\begin{equation*}
\int_{0}^{1}\dot{x}\left( t\right) \dot{v}\left( t\right)
dt+\int_{0}^{1}f\left( t,x\left( t\right) \right) v\left( t\right) dt=0
\end{equation*}%
for all $v\in H_{0}^{1}\left( 0,1\right) $. The classical solution to (\ref%
{par}) is then defined as a function $x:$ $\left[ 0,1\right] \rightarrow 
\mathbb{R}$ belonging to $H_{0}^{1}\left( 0,1\right) $ such that $\ddot{x}$
exists a.e. and $\ddot{x}\in L^{1}\left( 0,\pi \right) $. Since $f$ is
jointly continuous, then it is known from the Fundamental Theorem of the
Calculus of Variations, see \cite{Ma}, that $x$ is in fact twice
differentiable with classical continuous second derivative. Thus $x\in
H_{0}^{1}\left( 0,1\right) \cap C^{2}\left( 0,1\right) $.

Let $F\left( t,x\right) =\int_{0}^{x}f\left( t,s\right) ds$ for $\left(
t,x\right) \in \left[ 0,1\right] \times \mathbb{R}$. We link solutions to (%
\ref{par}) with critical points to a $C^{1}$ functional $J:H_{0}^{1}\left(
0,1\right) \rightarrow \mathbb{R}$ given by 
\begin{equation*}
J\left( x\right) =\frac{1}{2}\int_{0}^{1}\dot{x}^{2}\left( t\right)
dt+\int_{0}^{1}F\left( t,x\left( t\right) \right) dt.
\end{equation*}%
Let us examine $J$ for a while. Due to the continuity of $f$ functional $J$
is well defined. Recall that the norm in $H_{0}^{1}\left( 0,1\right) $ reads 
\begin{equation*}
\left\Vert x\right\Vert =\sqrt{\int_{0}^{1}\dot{x}^{2}\left( t\right) dt}.
\end{equation*}%
Then we see $\frac{1}{2}\int_{0}^{1}\dot{x}^{2}\left( t\right) dt=\frac{1}{2}%
\left\Vert x\right\Vert ^{2}$ is a $C^{1}$ functional by standard facts. Its
derivative is a functional on $H_{0}^{1}\left( 0,1\right) $ which reads 
\begin{equation*}
v\rightarrow \int_{0}^{1}\dot{x}\left( t\right) \dot{v}\left( t\right) dt.
\end{equation*}%
Concerning the nonlinear part we see that for any fixed $v\in
H_{0}^{1}\left( 0,1\right) $ (which is continuous of course) function $%
\varepsilon \rightarrow \int_{0}^{1}F\left( t,x\left( t\right) +\varepsilon
v\left( t\right) \right) dt$ (where the integral we can treat as the Riemann
one) due to the Leibnitz differentiation formula under integral sign is $%
C^{1}$ and the derivative of $\int_{0}^{1}F\left( t,x\left( t\right) \right)
dt$ is a functional on $H_{0}^{1}\left( 0,1\right) $ which reads 
\begin{equation*}
v\rightarrow \int_{0}^{1}f\left( t,x\left( t\right) \right) v\left( t\right)
dt
\end{equation*}%
if we recall that $F\left( t,x\right) =\int_{0}^{x}f\left( t,s\right) ds$.
Since the above is obviously continuous in $x$ uniformly in $v$ form unit
sphere, we see that $J$ is in fact $C^{1}.$

Recall also Poincar\'{e} inequality $\int_{0}^{1}x^{2}\left( t\right) dt\leq 
\frac{1}{\pi ^{2}}\int_{0}^{1}\dot{x}^{2}\left( t\right) dt$ and Sobolev's
one $\max_{t\in \left[ 0,1\right] }\left\vert x\left( t\right) \right\vert
\leq \int_{0}^{1}\dot{x}^{2}\left( t\right) dt$.

We sum up the assumptions on the nonlinear term in (\ref{par}) since in
order to get the above mentioned observations continuity of $f$ is
sufficient. We assume that\newline
\textit{\textbf{H1}} $f:\left[ 0,1\right] \times \mathbb{R} \rightarrow 
\mathbb{R}$ is a continuous function such that $f\left( t,0\right) \neq 0$
for $t\in \left[ 0,1\right]$;\newline
\textit{\textbf{H2}} $f$ is nondecreasing in $x$ for all $t\in \left[ 0,1%
\right] $

\begin{proposition}
Assume that \textbf{H1} and \textbf{H2} are satisfied. Then problem (\ref%
{par}) has exactly one nontrivial solution.
\end{proposition}

\begin{proof}
Firstly, we consider the existence part. Note that by Weierstrass Theorem
there exists $c>0$ such that 
\begin{equation*}
\left\vert f\left( t,0\right) \right\vert \leq c\text{ for all }t\in \left[
0,1\right] .
\end{equation*}%
Since $f$ is nondecreasing in $x$ \textit{\textbf{H2}} it follows that $F$
is convex. Since $F\left( t,0\right) =0$ for all $t\in \left[ 0,1\right] $
we obtain from the well known inequality 
\begin{equation}
F(t,x)=F(t,x)-F(t,0)\geq f\left( t,0\right) x\geq -\left\vert f\left(
t,0\right) x\right\vert   \label{aaa}
\end{equation}%
valid for any $x$ and all for all $t\in \left[ 0,1\right] $. We observe that
from (\ref{aaa}) we get 
\begin{equation}
F\left( t,x\right) \geq -c\left\vert x\right\vert \text{ for all }t\in \left[
0,1\right] \text{ and all }x\in \mathbb{R}.  \label{estimF}
\end{equation}%
Hence for any $x\in H_{0}^{1}\left( 0,1\right) $ we see by Schwartz and
Poincar\'{e} inequality 
\begin{equation*}
\int_{0}^{1}F\left( t,x\left( t\right) \right) dt\geq
-c\int_{0}^{1}\left\vert x\left( t\right) \right\vert dt\geq -\frac{c}{\pi }%
\left\Vert x\right\Vert .
\end{equation*}%
Therefore%
\begin{equation}
\begin{array}{l}
J\left( x\right) \geq \frac{1}{2}\left\Vert x\right\Vert ^{2}-\left\vert
c\right\vert \left\Vert x\right\Vert .%
\end{array}
\label{Ju_}
\end{equation}%
Hence from (\ref{Ju_}) we obtain that $J$ is coercive. Note that $\frac{1}{2}%
\left\Vert x\right\Vert ^{2}$ is obviously w.l.s.c. on $H_{0}^{1}\left(
0,1\right) $. Next, by the Arzela-Ascoli Theorem and Lebesgue Dominated
Convergence, see these arguments in full detail in \cite{Ma} in the proof of
Theorem 1.1 we see that $x\rightarrow \int_{0}^{1}F\left( t,x\left( t\right)
\right) dt$ is weakly continuous. Thus $J$ is weakly l.s.c. as a sum of a
w.l.s.c. and weakly continuous functionals. Since $J$ is $C^{1}$ and convex
functional it has exactly one argument of a minimum which is necessarily a
critical point and thus a solution to (\ref{par}). Putting $x=0$ in (\ref%
{par}) one see that we have a contradiction, so any solution is nontrivial.
\end{proof}

In order to get the existence of nontrivial solution to (\ref{par}) it would
suffice to assume that $f\left( t_{0},0\right) \neq 0$ for some $t_{0}\in %
\left[ 0,1\right] $ but since we need to impose same conditions on discrete
problem it is apparent that our assumption is more reasonable. Moreover,
there is another way to prove the weak lower semincontinuity of $J$, namely
show that $J$ is continuous. Then it is weakly l.s.c. since it is convex.
However, in proving continuity of $J$ on $H_{0}^{1}\left( 0,1\right) $ one
uses the same arguments.

\subsection{The discrete problem}

Now we turn the discretization of (\ref{par}), i.e. to problem (\ref{diffequ}%
). 
considered in the $n$-dimensional Hilbert space $E$ consisting of functions $%
x:\mathbb{N}(0,n)\rightarrow \mathbb{R}$ such that $x(0)=x(n)=0$. Space $E$
is considered with the following norm%
\begin{equation}
\left\Vert x\right\Vert =\left( \sum\limits_{k=1}^{n}|{\Delta }%
x(k-1)|^{2}\right) ^{\frac{1}{2}}.  \label{norm_operator}
\end{equation}%
We can also consider $E$ with the following norm 
\begin{equation*}
\left\Vert u\right\Vert _{0}=\left( \sum\limits_{k=1}^{n}|u(k)|^{2}\right) ^{%
\frac{1}{2}}.
\end{equation*}%
Since $E$ is finite dimensional there exist constants $c_{b}=\frac{1}{2}$ \
and $c_{a}=\left( \left( n-1\right) n\right) ^{1/2}$ such that 
\begin{equation}
c_{b}\left\Vert u\right\Vert \leq \left\Vert u\right\Vert _{0}\leq
c_{a}\left\Vert u\right\Vert \text{ for all }u\in E.  \label{c_a_c_b}
\end{equation}%
Solutions to (\ref{diffequ}) correspond in a $1-1$ manner to the critical
points to the following $C^{1}$ functional $\mathcal{I}:E\rightarrow \mathbb{%
R}$ 
\begin{equation*}
\mathcal{I}(x)=\sum\limits_{k=1}^{n}\tfrac{1}{2}|\Delta x(k-1)|^{2}+\frac{1}{%
n^{2}}\sum\limits_{k=1}^{n-1}F(\frac{k}{n},x(k))
\end{equation*}%
with $F$ defined as before. This means that 
\begin{equation*}
\frac{d}{dx}\mathcal{I}(x)=0\text{ if and only if }x\text{ satisfies (\ref%
{diffequ}).}
\end{equation*}%
Now we do not need to introduce the notion of the weak solution that is why
we have only one type of variational solution. We know that by the discrete
Schwartz Inequality by (\ref{estimF}) and by (\ref{c_a_c_b}) 
\begin{equation}
\begin{array}{l}
\mathcal{I}(x)\geq \frac{1}{2}\Vert x\Vert ^{2}-\frac{1}{n^{2}}\left\vert
c\right\vert \sqrt{n}\left( \sum\limits_{k=1}^{n-1}\left\vert x\left(
k\right) \right\vert ^{2}\right) ^{1/2} \\ 
\\ 
\geq \frac{1}{2}\Vert x\Vert ^{2}-\left\vert c\right\vert \frac{\sqrt{n-1}}{n%
}\left\Vert x\right\Vert \geq \frac{1}{2}\Vert x\Vert ^{2}-\left\vert
c\right\vert \left\Vert x\right\Vert .%
\end{array}
\label{relcoer}
\end{equation}%
Hence $\mathcal{I}(x)\rightarrow +\infty $ as $\Vert x\Vert \rightarrow
+\infty $ and we are in position to formulate the following

\begin{proposition}
\label{solvability_diff_equ}Assume that \textbf{H1}, \textbf{H2} hold. Then
problem (\ref{diffequ}) has exactly one nontrivial solution.
\end{proposition}

\subsection{Main result}

\begin{theorem}
\label{first convergence theorem}Assume that conditions \textbf{H1}, \textbf{%
H2} are satisfied. Then there exists $x\in H_{0}^{1}\left( 0,1\right) \cap
C^{2}\left( 0,1\right) $ which solves uniquely (\ref{par}) and for each $%
n\in \mathbb{N}$ there exists $x^{n}$ which solves uniquely (\ref{diffequ}).
Moreover, there exists a subsequence $x^{n_{m}}$ of $x^{n}$ such that
inequalities (\ref{ewa2}) are satisfied. 
\end{theorem}

\begin{proof}
We need to show that there exist two constants independent of $n$ such that
inequalities (\ref{ewa1}) hold. 
where $n_{0}$ is fixed. Then Lemma 9.2. from \cite{kelly} provides the
assertion of the theorem. In our argument we use some observations used in
the investigation of continuous dependence on parameters for ODE, see \cite%
{LedzewiczWalczak}. Fix $n$. By Proposition \ref{solvability_diff_equ},
there exists $x^{n}$ solving uniquely (\ref{diffequ}) and which is an
argument of a minimum to $\mathcal{I}$ such that it holds that $\mathcal{I}%
(x^{n})\leq \mathcal{I}(0)=0$. Thus relation (\ref{relcoer}) leads to the
inequality 
\begin{equation*}
\frac{1}{2}\Vert x^{n}\Vert\leq \left\vert c\right\vert \frac{\sqrt{n-1}}{n}.
\end{equation*}
Since $\max_{k\in \mathbb{N}(0,n)}\left\vert x^{n}\left( k\right)
\right\vert \leq \frac{\sqrt{n+1}}{2}\left\Vert x^{n}\right\Vert$ we get
that for all $k\in \mathbb{N}(0,n)$ 
\begin{equation*}
\left\vert x^{n}\left( k\right) \right\vert \leq 2\left\vert c\right\vert 
\frac{\sqrt{n-1}}{n}\frac{\sqrt{n+1}}{2}\leq \left\vert c\right\vert =N.
\end{equation*}
By Lemma 9.3 in \cite{kelly} we now obtain that there is a constant $Q$ such
that condition 
\begin{equation*}
n\vert \Delta x^{n}(k-1) \vert \leq Q \text{ and }\vert x^{n}(k) \vert \leq N
\end{equation*}
for all $k\in \mathbb{N}(0,n)$ and all $n\in \mathbb{N}$ is satisfied. This
means that the application of Lemma 9.2 from \cite{kelly} finishes the proof.
\end{proof}

\section{Final comments and examples}

In this section we provide the examples of nonlinear terms satisfying our
assumptions and we will investigate the possibility of replacing the
convexity assumption imposed on $F$ with some weaker requirement as well as
we comment on exisiting results in the literature.

Concerning the examples of nonlinear terms any nondecreasing $f$ is of order
bounded or unbounded, see 

\begin{enumerate}
\item[a)] $f\left( t,x\right) =g\left( t\right) \exp \left( x-t^{2}\right)$;

\item[b)] $f\left( t,x\right) =g\left( t\right) \arctan \left( x\right)$;

\item[c)] $f\left( t,x\right) =g\left( t\right) x^{3}+\exp \left(
x-t^{2}\right)$,
\end{enumerate}

where $g$ is any lower bounded continuous function with positive values. 

In view of remarks contained in \cite{Ma} functional $J$ can be written  
\begin{equation*}
J\left( x\right) =\left( \frac{1}{2}\int_{0}^{1}\dot{x}^{2}\left( t\right)
dt-\frac{a}{2\pi }\int_{0}^{1}x^{2}\left( t\right) dt\right) 
\end{equation*}%
\begin{equation*}
+\left( \int_{0}^{1}F\left( t,x\left( t\right) \right) dt+\frac{a}{2\pi }%
\int_{0}^{1}x^{2}\left( t\right) dt\right) .
\end{equation*}%
Then functional 
\begin{equation*}
x\rightarrow \left( \frac{1}{2}\int_{0}^{1}\dot{x}^{2}\left( t\right) dt-%
\frac{a}{2\pi }\int_{0}^{1}x^{2}\left( t\right) dt\right) 
\end{equation*}%
is strictly convex as long as $a\in \left( 0,1\right) $. Note that the first
eigenvalue of the differential operator $-\frac{d^{2}}{dt^{2}}$ with
Dirichlet boundary conditions on $\left[ 0,1\right] $ is $\frac{1}{\pi }$
(note this is the best constant in Poincar\'{e} inequality). Hence we can
relax convexity assumption $F$ by assuming that 
\begin{equation*}
x\rightarrow F\left( t,x\right) +\frac{a}{2\pi }x^{2}
\end{equation*}%
is convex for any $t\in \left[ 0,1\right] $. Then $F_{1}\left( t,x\right)
=F\left( t,x\right) +\frac{a}{2\pi }x^{2}$ satisfies (\ref{estimF}). 

The natural question arises if similar procedure is possible as far as the
discrete problem (\ref{diffequ}) is concerned. However there is one big
problem here since the first eigenvalue for $-\Delta ^{2}$ reads $\lambda
_{1}=2-2\cos \left( \frac{\pi }{n+1}\right) $ and of course $\lambda
_{1}\rightarrow 0$ as $n\rightarrow \infty $. This means that the above idea
would not work, since we cannot find $a$ for all $n$ idependent of $n$ (for
each $n$ such $a=a\left( n\right) $ exists ). 

A comparison with existing results is also in order. The only papers
concerning the existence of non-spurious solutions are \cite{rech1}, \cite%
{rachunkowa2}, \cite{thomsontisdell} which follow ideas developed in \cite%
{gaines} and which were mentioned already in the Introduction. We not only
use different methods, namely critical point theory, but also we are not
limited as far as the growth is concerned since in sources mentioned $f$ is
sublinear. However, we could not incorporate the derivative of $x$ into the
nonlinear term. This is not possible by variational approach but could be
made possible by connecting variational methods with Banach contraction
principle and it shows that the research concerning the existence of
non-spurious solutions with critical point approach can be further developed.

We cannot use sublinear growth as in sources mentioned since it does not
provide the inequality 
\begin{equation}
F\left( t,x\right) -F\left( t,0\right) \geq f\left( t,0\right) x\text{ for
all }t\in \left[ 0,1\right] \text{ and all }x\in \mathbb{R}.  \label{ccccc}
\end{equation}%
With our approach inequality (\ref{ccccc}) is essential in proving the
required estimations which lead to the existence of non-spurious solutions.
This is shown by the below remarks where direct calculations are performed.

The relevant growth condition reads\newline
\textit{\textbf{H2a} There exist constants} $a,b>0$ and $\gamma \in \left[
0,1\right) $ such that 
\begin{equation}
f\left( t,x\right) \leq a+b\left\vert x\right\vert ^{\gamma }\text{ for all }%
t\in \left[ 0,1\right] \text{ and all }x\in \mathbb{R}.
\label{cond_unboud_below}
\end{equation}%
By (\ref{cond_unboud_below}) for all $t\in \left[ 0,1\right] $ and all $x\in 
\mathbb{R}$ it holds 
\begin{equation*}
F\left( t,x\right) \leq a\left\vert x\right\vert +\frac{b}{\gamma +1}%
\left\vert x\right\vert ^{\gamma +1}.
\end{equation*}%
Since $F\left( t,x\right) \geq -\left\vert F\left( t,x\right) \right\vert $
we see by Schwartz, Holder and Poincar\'{e} inequality for any $x\in
H_{0}^{1}\left( 0,1\right) $ 
\begin{equation*}
\int_{0}^{1}F\left( t,x\left( t\right) \right) dt\geq -c_{1}\left\Vert
x\right\Vert -c_{2}\left\Vert x\right\Vert ^{\gamma +1},
\end{equation*}%
where $c_{1}=a$ and $c_{2}>0$ (the exact value of $c_{2}$ is not important
since $\gamma +1<2$ and functional $J$ is coercive disregarding of the value
of $c_{2}$. Then problem (\ref{par}) has at least one solution by the direct
method of the calculus of variations.

In order to consider problem (\ref{diffequ}) we need to perform exact
calculations since in this case, in view of the convergence Theorem \ref%
{first convergence theorem}, the precise values of constants are of utmost
importance. In case of \textit{\textbf{H2a}} from H\"{o}lder's inequality
and (\ref{c_a_c_b}) we get 
\begin{equation*}
\begin{array}{ll}
\sum\limits_{k=1}^{n-1}\vert u(k)\vert^{\gamma +1} & =\sum%
\limits_{k=1}^{n-1}\vert u(k)\vert^{\gamma +1} \cdot 1 \\ 
&  \\ 
& \leq \left( \sum\limits_{k=1}^{n-1}\vert u(k)\vert^{\gamma +1}\vert^{\frac{%
2}{\gamma +1} }\right) ^{\frac{\gamma +1}{2}}\left(
\sum\limits_{k=1}^{n-1}\vert1\vert^{\frac{1}{1- \frac{\gamma +1}{2}}}\right)
^{1-\frac{2}{\gamma +1}} \\ 
&  \\ 
& =\left( n-1\right) ^{\frac{1-\gamma }{2}}\left\Vert u\right\Vert
_{0}^{\gamma +1}\leq \left( \left( n-1\right) n\right) ^{\frac{\gamma +1}{2}%
}\left( n-1\right) ^{\frac{1-\gamma }{2}}\left\Vert u\right\Vert ^{\gamma +1}
\\ 
&  \\ 
& =\left( n-1\right) n^{\frac{\gamma +1}{2}}\left\Vert u\right\Vert ^{\gamma
+1}\leq \left( n-1\right) n\left\Vert u\right\Vert ^{\gamma +1}.%
\end{array}%
\end{equation*}
Thus 
\begin{equation*}
\frac{1}{n^{2}}\frac{b}{\gamma +1}\sum\limits_{k=1}^{n-1}\vert
u(k)\vert^{\gamma +1}\leq \frac{b}{\gamma +1}n^{\frac{\gamma -1}{2}%
}\left\Vert u\right\Vert ^{\gamma +1}
\end{equation*}

Hence by the above calculations and (\ref{relcoer}) we get for any $x\in E$ 
\begin{equation}
\mathcal{I}(x)\geq \frac{1}{2}\Vert x\Vert ^{2}-\left\vert a\right\vert 
\frac{\sqrt{n-1}}{n}\left\Vert x\right\Vert -\frac{b}{\gamma +1}n^{\frac{%
\gamma -1}{2}}\left\Vert x\right\Vert ^{\gamma +1}.  \label{rel_add_coer}
\end{equation}%
Thus $\mathcal{I}(x)\rightarrow +\infty $ as $\Vert x\Vert \rightarrow
+\infty .$ By Lemma 9.2. from \cite{kelly}, we need to show that (\ref{ewa1}%
) holds. Fix $n$. Since $\mathcal{I}(x^{n})\leq \mathcal{I}(0)=0$, the
relation (\ref{rel_add_coer}) leads to the inequality 
\begin{equation}
\frac{1}{2}\Vert x^{n}\Vert \leq \left\vert a\right\vert \frac{\sqrt{n-1}}{n}%
+\frac{b}{\gamma +1}n^{\frac{\gamma -1}{2}}\left\Vert x^{n}\right\Vert
^{\gamma }.  \label{ineq}
\end{equation}%
Since $\gamma <1$ we see $n^{\frac{\gamma -1}{2}}\rightarrow 0$. Thus there
is some $n_{0}$ that for all $n\geq n_{0}$ it holds $\frac{b}{\gamma +1}n^{%
\frac{\gamma -1}{2}}<\frac{1}{4}$. Take $n\geq n_{0}$. Let us consider two
cases, namely $\left\Vert x^{n}\right\Vert \leq 1$ and $\left\Vert
x^{n}\right\Vert >1$. In case $\left\Vert x^{n}\right\Vert >1$ we get from (%
\ref{ineq}) that 
\begin{equation*}
\frac{1}{2}\Vert x^{n}\Vert \leq \left\vert a\right\vert \frac{\sqrt{n-1}}{n}%
+\frac{1}{4}\left\Vert x^{n}\right\Vert 
\end{equation*}%
Recall $\max_{k\in \mathbb{N}(0,n)}\left\vert x^{n}\left( k\right)
\right\vert \leq \frac{\sqrt{n+1}}{2}\left\Vert x^{n}\right\Vert $ we get
that for all $k\in \mathbb{N}(0,n)$ 
\begin{equation*}
\left\vert x\left( k\right) \right\vert \leq 4\left\vert a\right\vert \frac{%
\sqrt{n-1}}{n}\frac{\sqrt{n+1}}{2}\leq 2\left\vert a\right\vert =N.
\end{equation*}%
For the case $\left\Vert x^{n}\right\Vert \leq 1$ we however we cannot
proceed without (\ref{ccccc}). The reason is what while on space $E$
disregarding of $n$ the sequence is norm bounded by $1$ (uniformely in $n$)
in norm given by (\ref{norm_operator}), this is not the case with the
max-norm where it is unbounded as $n\rightarrow \infty $.

\end{document}